\documentclass[12pt,a4paper,final]{amsart}
\usepackage{version}
\usepackage[OT1]{fontenc}    %
\usepackage{type1cm}         %
\usepackage{amsthm}          %
\usepackage[dvips]{graphicx} 
\usepackage[bf]{caption}     
\usepackage{psfrag}          

\usepackage{verbatim}        
\usepackage[notcite,notref]{showkeys}

\newtheorem{theorem}{Theorem}[section]
\newtheorem{lemma}[theorem]{Lemma}

\theoremstyle{definition}

\newtheorem{example}[theorem]{Example}

\theoremstyle{remark}
\newtheorem{remark}[theorem]{Remark}

\numberwithin{equation}{section}

\DeclareMathOperator{\dimp}{dim_p}
\DeclareMathOperator{\dimH}{dim_H}
\DeclareMathOperator{\spt}{spt}
\DeclareMathOperator{\por}{por}
\DeclareMathOperator{\dist}{dist}
\DeclareMathOperator{\dimloc}{\overline{dim}_{loc}}
\DeclareMathOperator*{\essinf}{ess\phantom{:}inf}

\begin{document}

\title{Packing dimension of mean porous measures}

\author[DB,EJ,MJ,AK,TR,SS,VS]{D. Beliaev, E. J\"arvenp\"a\"a, 
        M. J\"arvenp\"a\"a, A. K\"aenm\"aki, T. Rajala,\\ S. Smirnov and 
        V. Suomala} 

\address{DB: Department of Mathematics\\ 
         Princeton University\\
         Fine Hall, Washington Road
         Princeton NJ 08544-1000\\
         USA} 
\address{EJ,MJ,AK,TR and VS: Department of Mathematics and Statistics \\
         P.O. Box 35 (MaD) \\
         FIN-40014 University of Jyv\"askyl\"a \\
         Finland}
\address{SS: Department of Mathematics\\
         University of Geneva\\
         2-4 rue du Li\`evre, Case postale 64\\
         1211 Gen\`eve 4\\
         Switzerland}

\email{dbeliaev@math.princeton.edu}
\email{esaj@maths.jyu.fi}
\email{amj@maths.jyu.fi}
\email{antakae@maths.jyu.fi}
\email{tamaraja@maths.jyu.fi}
\email{Stanislav.Smirnov@math.unige.ch}
\email{visuomal@maths.jyu.fi}

\keywords{Porosity, mean porosity, packing dimension, Radon measures}

\subjclass[2000]{28A75, 28A80}

\thanks{EJ, MJ, AK, TR and VS acknowledge the support of the Academy of Finland
(projects \# 211229 and \# 114821) and the Centre of Excellence in Analysis
and Dynamics Research. EJ also thanks the hospitality of the 
University of Geneva. DB, EJ and SS acknowledge the support of the Swiss 
National Science Foundation. TR appreciates the financial support of Vilho, 
Yrj\"o and Kalle V\"ais\"al\"a foundation. Finally, we all thank the referee
for valuable comments clarifying the exposition.}

\begin{abstract}
We prove that the packing dimension of any mean porous Radon measure 
on $\mathbb R^d$ may be estimated from above by a function which depends on
mean porosity. The upper bound tends to $d-1$ as mean porosity tends
to its maximum value. This result was stated in \cite{BS}, and in a weaker form
in \cite{JJ1}, but the proofs are not correct. Quite surprisingly, it turns out
that mean porous measures are not necessarily approximable by 
mean porous sets. We verify this by constructing an example of 
a mean porous measure $\mu$ on $\mathbb R$ such that $\mu(A)=0$ for all 
mean porous sets $A\subset\mathbb R$. 
\end{abstract}

\maketitle

\section{Introduction}\label{intro}

Intuitively, it seems obvious that if a set contains relatively large holes 
at all small scales then the dimension of the set should be smaller than that 
of the ambient space. This observation was generalized and made 
into a quantitative form 
by Mattila \cite{Ma1} in terms of a concept called porosity which describes 
the sizes of holes at all small scales (for the definition see Section 
\ref{definitions}). Mattila proved that if the porosity of a subset of
$\mathbb R^d$ is close to its maximum value $\tfrac 12$ then its Hausdorff 
dimension cannot be much bigger than $d-1$. The correct asymptotic behaviour 
was established by Salli \cite{S}. He also showed that Hausdorff dimension 
may be replaced by packing dimension, and moreover, by box counting 
dimension under the assumption that the set is uniformly porous.

The above mentioned result for Hausdorff dimension fails if the set contains
large holes only at sequences of arbitrarily small scales; there are examples 
of such sets in $\mathbb R^d$ with Hausdorff dimension $d$ \cite{Ma2}. 
Nevertheless, the assumption that the set has relatively large holes at all 
small scales may be weakened to obtain an upper bound for Hausdorff dimension, 
or more generally, for packing dimension. In fact, it is 
sufficient to suppose that a certain percentage of scales contains holes. 
This leads to the concept of mean porosity (for the definition see 
Section \ref{definitions}). Dimensional properties of such sets were 
considered by Koskela and Rohde \cite{KR} in the case of small mean porosity, 
and by Beliaev and Smirnov in the case of large one. 
For other related results, see \cite{JJKS}.
 
In this paper, the emphasis is given to packing dimensions of mean porous 
measures
(for the definition see Section \ref{definitions}). Porous measures were 
introduced by Eckmann, J\"ar\-ven\-p\"a\"a and J\"arvenp\"a\"a \cite{EJJ} 
whilst the analogue of Mattila's result was verified for porous measures in 
\cite{JJ1}. Note that in \cite{JJ1} the results are claimed 
for packing dimension but the argument works only for Hausdorff 
dimension as explained in \cite{JJ2}. 

The study of mean porous measures was pioneered by Beliaev and Smirnov
\cite{BS}. In \cite{BS} the proof of the statement that the same upper bound 
which is valid for packing dimensions of mean porous sets holds for 
mean porous measures as well is based on a proposition claiming that 
mean porous measures 
are approximable by mean porous sets. However, this is not the case: 
in  Theorem \ref{thm} we construct a mean porous measure $\mu$ such 
that all mean porous sets have zero $\mu$-measure. The main purpose of this
paper is to develop a new method to show that the statements of 
\cite{BS,JJ1} are true (see Theorem \ref{theorem:pdimofporo}) even though the
proofs are not correct. 

The paper is organized as follows: In Section \ref{definitions} we discuss 
the basic concepts. Section \ref{mainsection} is dedicated to the proof 
of our main result. Besides this, we illustrate by an example that 
the upper bound which we obtain is asymptotically the best possible one. 
Finally, in the last section we construct an example of a mean porous measure 
which is not approximable by mean porous sets.

\section{Basic concepts}\label{definitions}

In this section, we give the basic definitions used throughout the paper.
Intuitively, the porosity of a set gives for all small scales the relative
radius of the largest ball which fits into a reference ball centered at the
set and which does not intersect the set.
Let $A\subset\mathbb{R}^d$. For all $x\in\mathbb R^d$ and $r>0$, we define 
\[
\por(A,x,r)=\sup\{\alpha\ge0\,:\,B(y,\alpha r)\subset B(x,r)\setminus A
\text{ for some }y\in\mathbb{R}^d\}.
\]
Here $B(x,r)$ is the closed ball with centre at $x$ and radius $r$.
Clearly, $0\le\por(A,x,r)\le\tfrac 12$ for all $x\in A$. Given 
$0\le\alpha\le\tfrac 12$, the set $A$ is said to be $\alpha$-porous at 
$x$ if
\[
\liminf_{r\to0}\por(A,x,r)\geq\alpha.
\]
Moreover, $A$ is 
$\alpha$-porous if it is $\alpha$-porous at every point $x\in A$.

For measures, the corresponding concepts are defined as follows:
Let $\mu$ be a Radon measure on $\mathbb{R}^d$. For all $x\in\mathbb{R}^d$ and 
for all positive real numbers $r$ and $\varepsilon$, set
\begin{align*}
\por(\mu,x,r,\varepsilon)
 =&\sup\{\alpha\ge0\,:\,\text{there is $z\in\mathbb{R}^d$ such that }\\
  & B(z,\alpha r)\subset B(x,r)\text{ and }\mu(B(z,\alpha r))
     \le\varepsilon\mu(B(x,r))\}.
\end{align*}
Given  $\alpha\ge0$, the measure $\mu$ is $\alpha$-porous at 
a point $x\in\mathbb{R}^d$ if 
\[
\lim_{\varepsilon\to0}
\liminf_{r\to0}\por(\mu,x,r,\varepsilon)\ge\alpha.
\]
The order of taking limits is important here: if we changed it we would
obtain the porosity of $\spt\mu$, the support of $\mu$.
Finally, the measure $\mu$ is $\alpha$-porous if there is 
$A\subset\mathbb{R}^d$ with $\mu(A)>0$ such that $\mu$ is $\alpha$-porous 
at every point $x\in A$. It is not difficult to see that in this case 
$0\le\alpha\le\tfrac 12$. For more information on porosity of measures,
see \cite{EJJ}.

Larger classes of mean porous sets and measures are obtained by demanding that 
a certain percentage of scales - not necessarily all small ones - are porous.
Given $\alpha\ge0$ and a positive integer $j$, the set $A$ is
$\alpha$-porous for scale $j$ at a point $x\in\mathbb{R}^d$ 
whenever $\por(A,x,2^{-j})\geq \alpha$. For $0<p\leq 1$, the set $A$ 
is called mean $(\alpha,p)$-porous at a point $x\in\mathbb{R}^d$ if
\[
\liminf_{i\rightarrow\infty}\frac{\#\{1\leq j\leq i\,:\,
        \por(A,x,2^{-j})\geq\alpha\}}{i}\geq p.
\]
Here the cardinality of a set is denoted by $\#$. We say that $A$ is 
mean $(\alpha,p)$-porous if it is mean $(\alpha,p)$-porous at every point 
$x\in A$. The measure $\mu$, in turn, is mean $(\alpha,p)$-porous at $x$ if  
\[
\lim_{\varepsilon\to0}\liminf_{i\rightarrow\infty}\frac{\#\{1\leq j\leq i\,:\,
    \por(\mu,x,2^{-j},\varepsilon)\geq \alpha\}}{i}\geq p.
\]
Finally, $\mu$ is mean $(\alpha,p)$-porous if there is 
$A\subset\mathbb{R}^d$ with $\mu(A)>0$ such that $\mu$ is mean 
$(\alpha,p)$-porous at all points $x\in A$.

The packing dimension, $\dimp$, of a Radon measure $\mu$ on $\mathbb{R}^d$ 
is defined in terms of local dimensions as follows:
\[
\dimp\mu=\mu\text{-}\essinf_{x\in\mathbb R^d}\dimloc\mu(x)
\]
where
\[
\dimloc\mu(x)=\limsup_{r\downarrow 0}\frac{\log(\mu(B(x,r)))}{\log r}
\]
and $\mu\text{-}\essinf$ means the essential infimum with respect to $\mu$.
Equivalently, the packing dimension of $\mu$ is given by means of packing
dimensions of Borel sets with positive $\mu$-measure \cite{C}:
\[  
\dimp\mu=\inf\{\dimp A\,:\, A\text{ is a Borel set with }\mu(A)>0\}.
\] 

\begin{remark}\label{upperdim} Replacing the essential infimum by the 
essential supremum in the above definition leads to the concept of  
upper packing dimension. Using the fact that restricting a measure will not
decrease the porosity \cite{EJJ}, we see that Theorem \ref{theorem:pdimofporo}
is valid for the upper packing dimension of $\mu$ as well provided 
that $\mu$ is mean porous $\mu$-almost everywhere.
\end{remark}

\section{Packing dimension of measures with large mean porosity}
\label{mainsection}

In this section we prove the following packing dimension estimate for
mean porous measures:

\begin{theorem}\label{theorem:pdimofporo} Let $0\le\alpha\le\tfrac12$ and
$0<p\le1$.
There exists a constant $C$ depending only on $d$ such that for all
mean $(\alpha,p)$-porous Radon measures $\mu$ on $\mathbb{R}^d$ we have
\[\dimp\mu\le d - p +\frac{C}{\log(\tfrac 1{1-2\alpha})}.\]
\end{theorem}

At the end of this section, we give a construction (Example 
\ref{asymptoticallybest}) which indicates that the upper bound of Theorem 
\ref{theorem:pdimofporo} is asymptotically the best possible one as $\alpha$ 
tends to $\tfrac 12$. The proof of Theorem \ref{theorem:pdimofporo} is given as
a series of lemmas. The first one serves as a key tool in the proof of 
our main result. We use the symbol $r_Q$ for the side-length of a cube 
$Q\subset\mathbb R^d$. 

\begin{lemma}\label{lemma:dimensionlemma} Let $m,i_0\in\mathbb{N}$ and $D>0$.
Let $\mu$ be a Radon measure on $\mathbb{R}^d$ with 
$0<\mu(\mathbb{R}^d)<\infty$. Assume that all 
disjoint collections $\mathcal{Q}$ of half-open $2^m$-adic cubes with 
side-length at most $2^{-mi_0}$ have the following property: 
for all $Q\in\mathcal Q$ there is $0<\tau(Q)<D$ such that
\[
\sum_{Q\in\mathcal{Q}}r_{Q}^{\tau(Q)}\mu(Q)^{1-\frac{\tau(Q)}{D}} <\mu(\mathbb{R}^d).
\] 
Then
\[
\dimp\mu \le D.
\]
\end{lemma}

\begin{proof}
Suppose to the contrary that $\dimp\mu>D$. Then
$\dimloc\mu(x)>D$ for $\mu$-almost every $x\in\mathbb{R}^d$.
Recall from \cite[Lemma 2.3]{C} that for $\mu$-almost all $x\in\mathbb R^d$
the local dimension $\dimloc\mu(x)$ may be calculated using half-open 
$2^m$-adic cubes 
containing $x$ instead of balls $B(x,r)$. Hence, for $\mu$-almost every
$x\in\mathbb{R}^d$ we may choose a cube $Q_x$ containing $x$ and with 
side-length 
$2^{-mi}$ for some $i>i_0$ such that $\mu(Q_x)<r_{Q_x}^D$.
Let $\mathcal{Q}$ be a disjoint collection of such cubes
covering $\mu$-almost all points of $\mathbb{R}^d$. Then
\[
\mu(\mathbb{R}^d)=\sum_{Q\in\mathcal{Q}}\mu(Q)^{\frac{\tau(Q)}{D}}
     \mu(Q)^{1-\frac{\tau(Q)}{D}}
   <\sum_{Q\in\mathcal{Q}}r_Q^{\tau(Q)}\mu(Q)^{1-\frac{\tau(Q)}{D}}
   <\mu(\mathbb{R}^d),
\]
which is a contradiction.
\end{proof}

The following lemma shows that when cubes $Q$ are of the same size we can 
approximate the sums $\sum_Qr_{Q}^\tau\mu(Q)^{1-\frac{\tau}{D}}$ from above by 
distributing the measure evenly on the cubes. 

\begin{lemma}\label{lemma:worstcase} Let $D>0$. Assume that  $0<\tau<D$ and
$Q^1,\ldots,Q^N\subset\mathbb{R}^d$ are disjoint cubes with side-length $r$.
Then for any Radon measure $\mu$ on $\mathbb R^d$
\[
 \sum_{j=1}^Nr^\tau\mu(Q^j)^{1-\frac{\tau}{D}}\le N^{\frac{\tau}{D}}r^{\tau}
   \mu(\bigcup_{j=1}^NQ^j))^{1-\frac{\tau}{D}}\,\,\Big(=N r^\tau
    \big(\frac{\mu(\bigcup_{j=1}^NQ^j)}{N}\big)^{1-\frac{\tau}{D}}\Big). 
\] 
\end{lemma}

\begin{proof}
Since the cubes $Q^1,\ldots,Q^N$ are of the same size, the claim follows 
directly from H\"older's inequality: 
\begin{align*}
\sum_{j=1}^Nr^\tau\mu(Q^j)^{1-\frac{\tau}{D}} =
r^\tau\sum_{j=1}^N\mu(Q^j)^{1-\frac{\tau}{D}}\le
N^{\frac{\tau}{D}}r^\tau\Big(\sum_{j=1}^N\mu(Q^j)\Big)^{1-\frac{\tau}{D}}. 
\end{align*}
\end{proof}

The proof of the next lemma is based on straightforward geometric
arguments. The boundary of a set $A\subset\mathbb R^d$ is denoted by
$\partial A$.

\begin{lemma}\label{lemma:convexlike} Let $k\in\mathbb N$ and
let $Q\subset\mathbb R^d$ be a dyadic cube. If $B_1,B_2,\ldots B_n$ are closed
balls with radii at least $\sqrt{d}r_Q$, then 
$\partial(Q\setminus\bigcup_{i=1}^{n}B_i)$ may be
covered by $c2^{k(d-1)}$ dyadic cubes of side length $2^{-k}r_Q$. Here
$c=c(d)$ is a positive and finite constant depending only on $d$.  
\end{lemma}

\begin{proof} Let $B_i=B(x_i,r_i)$ and 
$\Gamma=\partial(Q\setminus\bigcup_{i=1}^{n}B_i)$. We may assume that 
$B_i\cap Q\ne\emptyset$ and $x_i\not\in Q$ since $r_i\ge\sqrt dr_Q$. Denote the 
faces of $Q$ by $F_1,\dots,F_{2d}$, and divide the balls $B_i$ into $2d$ 
disjoint sets $A_j$ such that $B_i\in A_j$ provided $F_j$ is the closest face 
to $x_i$. Here the distance is measured from the centre of a face. (If there 
are several faces which are equally close choose one of those.) Fix 
$j\in\{1,\dots,2d\}$ and consider the part of $\Gamma$ determined by $A_j$.  
Since $r_i\ge\sqrt dr_Q$ we have that $F_j\subset B_i$ if $x_i$ is sufficiently
close to $F_j$. Hence, there is a constant $\gamma>0$ such that the angle 
between the normal of $F_j$ and the tangent plane of $B_i$ is larger than 
$\gamma$ at any point in $\Gamma\cap B_i$ for all $B_i\in A_j$. The minimum 
point for this angle is obtained in the 
following manner: Consider a vertex $v$ of $Q$ contained in $F_j$. Denote by
$v'$ the vertex of $F_j$ which is opposite to $v$, that is, the 
line segment determined  by $v$ and $v'$ is a diagonal of $F_j$. Let $L$ be 
the line determined by the diagonal of $Q$ containing $v$. Move along $L$ 
away from $Q$ up to the point $y$ where the distance to $v'$ equals 
$\sqrt d r_Q$.  Now $y$ is a minimum point which determines $\gamma$. We 
conclude that there is a bi-Lipschitz 
injection from $\cup_{B_i\in A_j}B_i\cap\Gamma$ to $F_j$ such that the Lipschitz
constants depend only on $\gamma$. This gives the claim.
\end{proof}

Before stating the rest of the auxiliary results, we introduce the notation 
we need throughout the remaining part of this section.

Consider the smallest integer $l$ such that $4\sqrt{d}\leq2^l$. For 
$\tfrac{15}{32}<\alpha<\tfrac12$, let $k=k(\alpha)$ be the unique positive 
integer for which 
\begin{equation}\label{eq:kdef}
\sqrt{d}2^{-k-1}\leq (1-2\alpha)2^l<\sqrt{d} 2^{-k}.
\end{equation}
For all $i\in\mathbb N$, the collection of all 
half-open $2^k$-adic cubes of side-length $2^{-ki}$ is denoted by 
$\mathcal{Q}^i$. Moreover, if $Q\in\mathcal{Q}^{i}$ and 
$Q'\in\mathcal{Q}^{i+n}$, we use the notation $Q'\prec_n Q$ provided that 
$Q'\subset Q$. This is simplified to $Q'\prec Q$
in the case $n=1$.

Let $\mu$ be a Radon measure on $\mathbb R^d$. Given $\varepsilon>0$, we call 
a cube $Q\in\mathcal{Q}^{i+1}$ porous provided that
\[
\por(\mu,x,2^{-ki+l},\varepsilon)\ge\alpha
\] 
for some $x\in Q$. For all $Q\in\mathcal{Q}^{i}$, set
\[
Q_{por}=\bigcup\{Q'\prec Q\,:\,Q'\text{ is porous}\}.
\]

Finally, for $x\in\mathbb{R}^d$ let 
$Q^{i}_{x}\in\mathcal{Q}^i$ be the unique cube containing $x$. If $c>0$,
$cQ$ is the cube obtained from a cube $Q$ by magnifying by the factor $c$ 
with respect to the centre of $Q$.

One of the fundamental and most useful structural properties of porous sets 
is the following: if the porosity of $A\subset\mathbb{R}^d$ is close to 
$\tfrac12$, then locally inside each ball with radius $r$ the set $A$ is
contained in a small neighbourhood of some $(d-1)$-dimensional surface
with $\mathcal H^{d-1}$-measure comparable to $r^{d-1}$. For more precise
statements of this type, see e.g. \cite{S} or \cite{JJKS}. 
In the following lemma which is a slight improvement of 
\cite[Lemma 2.2]{JJ1} we translate this fact into the language of mean
porous measures. It states that every cube $Q$ may be divided into three 
parts $Q=E\cup P\cup J$, where $E$ has small measure, $P$ is a small 
neighbourhood of some $(d-1)$-dimensional surface and $J$ contains no porous
points.

\begin{lemma}\label{lemma:boundary}
Let $\mu$ be a Radon measure on $\mathbb{R}^d$, $i\in\mathbb N$ and 
$\varepsilon>0$. Then any cube
$Q\in\mathcal{Q}^i$ may be divided into three parts 
\[
Q=E\cup P\cup J,
\] 
where $\mu(E)\leq N\varepsilon \mu((1+2^{l+1})Q)$, $J\cap Q_{por}=\emptyset$,
and $P$ can be covered by at most $c2^{k(d-1)}$ cubes $Q'\prec Q$.
Here $N=N(\alpha,d)$ and $c=c(d)$ are positive and finite constants.
\end{lemma}

\begin{proof} For any $Q'\prec Q$ with $Q'\subset Q_{por}$, there is 
$x\in Q'$ such that
\[
\por(\mu,x,2^{-ki+l},\varepsilon)\ge\alpha.
\] 
Hence, given $0<\alpha'<\alpha$ with $(1-2\alpha')2^l<\sqrt d2^{-k}$, we 
find a ball $B_{Q'}$ of radius 
$\alpha'2^{-ki+l}$ such that
\begin{equation}\label{eq:Bx}
B_{Q'}\subset B(x,2^{-ki+l})\subset (1+2^{l+1})Q\,\,\text{ and }\,\,
\mu(B_{Q'})\leq\varepsilon\mu((1+2^{l+1})Q).
\end{equation}

Denoting by $\dist(a,A)$ the distance from a point $a\in\mathbb R^d$ to a set 
$A\subset\mathbb R^d$, define
\begin{align*}
 E&=\bigcup_{\substack {Q'\prec Q,\\ Q'\subset Q_{por}}}B_{Q'}\cap Q,\\
 P&=\{x\in Q\,:\,\dist\big(x,\partial(Q\setminus E)\big)<2^{-k(i+1)+l}\},
  \text{ and}\\
 J&=Q\setminus(E\cup P).
\end{align*}
It is evident from \eqref{eq:Bx} that 
\[
\mu(E)\leq 2^{kd}\varepsilon\mu((1+2^{l+1})Q)=N(\alpha,d)\varepsilon
   \mu((1+2^{l+1})Q).
\]  
On the other hand, since $\alpha'>\frac 7{16}$, we have that 
$\alpha'2^{-ki+l}>\sqrt d r_Q$, and so
it follows easily from Lemma \ref{lemma:convexlike} that 
$P$ may be covered by at most $(1+2^{l+1})^dc'2^{k(d-1)}=c(d)2^{k(d-1)}$ cubes 
$Q'\prec Q$ where $c'=c'(d)$ is the constant of Lemma \ref{lemma:convexlike}. 

It remains to show that $J\cap Q_{por}=\emptyset$, that is,
\begin{equation}\label{eq:subset}
Q_{por}\subset E\cup P. 
\end{equation}
Consider $y\in Q_{por}\setminus E$ and choose $Q'\prec Q$ with 
$y\in Q'\subset Q_{por}$. Let $x\in Q'$ and $B_{Q'}$ be as in \eqref{eq:Bx}. 
By the choice of $\alpha'$ we have
\[
\dist(y,B_{Q'})\leq|y-x|+(1-2\alpha')2^{-ki+l}<2\sqrt{d}2^{-k(i+1)}
   \leq 2^{-k(i+1)+l}
\]
which gives $\dist(y,\partial(Q\setminus E))<2^{-k(i+1)+l}$, and therefore
$y\in P$. This completes the proof of \eqref{eq:subset}.
\end{proof}

Next lemma is in the core of the proof of Theorem \ref{theorem:pdimofporo}. 
It shows that summing over porous
subcubes of a cube $Q$ gives a small factor which is decreasing 
exponentially under iteration. This is needed when proving that the assumptions
of Lemma \ref{lemma:dimensionlemma} are valid.

For the purpose of formulating our key lemma, we define weights
$\beta(Q)$ for cubes $Q\in\mathcal{Q}^i$ as follows: Suppose that $\mu$ is
a Radon measure on $\mathbb R^d$ and $0<D<d$. Set
\begin{equation}\label{defC}
C=C(d)=\max\{c,2d2^l\}
\end{equation}
where $c=c(d)$ is as in Lemma \ref{lemma:boundary}. If
$Q\in\mathcal{Q}^i$ for some $i\in\mathbb{N}$, define
\begin{equation}\label{weights}
\beta(Q)=\frac 13\begin{cases}
       C^{-\frac 12}2^{-\frac k2(d-1-D)}&\text{ if $Q$ is porous}\\
       2^{-\frac k2(d-D)}&\text{ otherwise}.
     \end{cases}
\end{equation}
For all $n\in\mathbb N$ with $n\ge2$, let 
$\varepsilon_0=\varepsilon_0(d,\alpha,D,n)>0$ be the unique real number 
satisfying
\begin{equation}\label{eq:R1}
R(\varepsilon_0)=
\frac{5}{18}C^{-\frac 12}2^{\frac k2}, 
\end{equation}
where for all $\varepsilon>0$
\begin{equation}\label{eq:R2}
\begin{split}
&R(\varepsilon)=R(\varepsilon,n)=\\
&\frac{n-1}{3}(\varepsilon N)^{\frac 12}C^{-\frac 12} 2^{\frac k2}
2^{k(n-1)(d-\frac D2)}\max\{1,\frac13 C^{-\frac 12}
2^{-\frac k2(d-1-D)} \}^{n-1} 
\end{split}
\end{equation}
and $N=N(\alpha,d)$ is as in Lemma \ref{lemma:boundary}.

Before formulating our lemma we make one more remark: From now on, we assume
that $\tfrac{15}{32}<\alpha<\tfrac12$ is so close to $\tfrac12$ that
\begin{equation}\label{eq:k}
k=k(\alpha)\geq\frac{\log C}{\log 2}.
\end{equation}
Thus $C^{-\frac 12}2^{-\frac k2(d-1-D)}\geq 2^{-\frac k2(d-D)}$,
giving 
\begin{equation}\label{trivialestimate}
\beta(Q)\leq \tfrac13 C^{-\frac 12}2^{-\frac k2(d-1-D)}
\end{equation}
for all $Q$. This fact will be used repeatedly in the proof of Lemma 
\ref{lemma:uniformestimate}.

For any $Q\in\mathcal{Q}^i$ and $j=0,1,2,\ldots,i$, we denote by $Q_{j}$
the unique cube in $\mathcal{Q}^{j}$ for which $Q\subset Q_{j}$. Clearly, 
$Q_i\subset Q_{i-1}\subset\dots\subset Q_1\subset Q_0$ and $Q_i=Q$.

\begin{lemma}\label{lemma:uniformestimate}
Let $\mu$ be a Radon measure on $\mathbb{R}^d$, $0<D<d$ and
$n\in\mathbb{N}$ with $n\ge2$. Then for all 
$i\in\mathbb N$ and $Q\in\mathcal Q_i$ we have
\begin{equation}\label{claim1}
\sum_{Q' \prec_n Q}\bigl(\prod_{j=1}^{n}\beta(Q'_{i+j})\bigr)
   r_{Q'}^{\frac D2}\mu(Q')^{\frac 12} 
\leq C^{-\frac 12}2^{\frac k2} r_{Q}^{\frac D2}\mu(Q)^{\frac 12}.
\end{equation} 
Moreover, if $\mathcal{Q}$ is any finite collection of disjoint $2^{kn}$-adic 
cubes and $\mu$ is a Radon measure on 
$\mathbb R^d$ such that $\spt\mu\subset\mathopen[0,1\mathclose]^d$, then
\begin{equation}\label{claim2}
\sum_{Q\in\mathcal{Q}}\bigl(\prod_{j=0}^{N_Q-1}C^{\frac 12}2^{-\frac k2}
  \prod_{i=jn+1}^{(j+1)n}\beta(Q_i)\bigr)r_{Q}^{\frac D2}\mu(Q)^{\frac 12}
  \leq r_{\mathopen[0,1\mathclose]^d}^{\frac D2}\mu(\mathbb{R}^d)^{\frac 12}, 
\end{equation}
where $N_Q\in\mathbb N$ such that $r_Q=2^{-knN_Q}$.
\end{lemma}

\begin{proof}
The proof is based on Lemma \ref{lemma:boundary}. The problematic part in 
Lemma \ref{lemma:boundary} is the factor $\mu((1+2^{l+1})Q')$. If $Q'$ is close
to the boundary of $Q$ the expanded cube $(1+2^{l+1})Q'$ will not be a subset
of $Q$, and therefore we are unable to estimate the sum in terms of 
$\mu$-measure of $Q$. This
problem is overcome by dividing the subcubes of $Q$ into two parts depending
on there distance to the boundary of $Q$. This leads to the use of $n$: the
larger $n$ we take, the better estimates we have. For $n=1$ we cannot utilize
porosity at all.
 
Claim \eqref{claim2} is a direct consequence of repeated applications of 
\eqref{claim1}. In what follows
we will prove \eqref{claim1} by induction. Divide the cube 
$Q\in\mathcal{Q}^i$ into two regions of subcubes, the boundary region
$\mathcal{Q}_B\subset\mathcal{Q}^{i+1}$ and the interior one 
$\mathcal{Q}_I\subset\mathcal{Q}^{i+1}$, as follows
\[
\mathcal{Q}_B = \{Q' \prec Q~:~\dist(Q', \partial Q) \le 2^{-(i+1)k+l}\}
\] 
and
\[
\mathcal{Q}_I = \{Q' \prec Q~:~\dist(Q', \partial Q) > 2^{-(i+1)k+l}\}.
\] 

For $n=2$, we estimate the sum in these subcubes in the following manner: 
Let $R(\varepsilon)=R(\varepsilon,2)$ and $\varepsilon_0$ be as in 
\eqref{eq:R1} and \eqref{eq:R2}. Take
$Q''\in\mathcal{Q}_I$. Defining $Q''_{por}$ in terms of $\varepsilon_0$, we 
obtain from definition \eqref{weights} and Lemma \ref{lemma:worstcase} that
\begin{equation}\label{eq:nonporoestimate}
\begin{split}
&\sum_{\substack{Q' \prec Q''\\Q'\cap Q''_{por}=\emptyset}}
   \beta(Q')r_{Q'}^{\frac D2}\mu(Q')^{\frac 12}\\
&\le\frac 13 2^{-\frac k2(d-D)} 2^{\frac{kd}{2}}
   2^{-\frac k2(i+2)D}\mu(Q'')^{\frac 12}\\ 
&=\frac13 2^{-\frac k2(i+1)D}\mu(Q'')^{\frac 12}.
\end{split}
\end{equation}
To estimate the sum over $Q'\subset Q''_{por}$, we apply Lemma 
\ref{lemma:boundary} to $Q''$. Note that the part $J$ is now absent, and 
therefore  the sum 
is divided into two parts determined by $E$ and $P$. Using Lemma 
\ref{lemma:worstcase} in both parts and the fact that 
$(1+2^{l+1})Q''\subset Q$, we have recalling \eqref{defC} 
\begin{equation}\label{eq:poroestimate}
\begin{split}
&\sum_{\substack{Q' \prec Q''\\ Q'\subset Q''_{por}}}
  \beta(Q')r_{Q'}^{\frac D2}\mu(Q')^{\frac 12} \\
&\le\sum_{\substack{Q' \prec Q''\\ Q'\subset Q''_{por},\,Q'\cap E\ne\emptyset}}
  \beta(Q')r_{Q'}^{\frac D2}\mu(Q')^{\frac 12} +
  \sum_{\substack{Q' \prec Q''\\ Q'\subset Q''_{por},\,Q'\cap P\ne\emptyset}}
  \beta(Q')r_{Q'}^{\frac D2}\mu(Q')^{\frac 12}\\
&\leq \frac 13 C^{-\frac 12}2^{-\frac k2(d-1-D)}\Big( 2^{\frac{kd}2}
  2^{-\frac k2(i+2)D}\big(N\varepsilon_0\mu((1+2^{l+1})Q'')\big)^{\frac 12}\\  
&+\big(c2^{k(d-1)}\big)^{\frac 12}2^{-\frac k2(i+2)D}\mu(Q'')^{\frac 12}\Big)\\
&\le \frac 13 C^{-\frac 12}2^{-\frac k2(d-1-D)}\Big(2^{\frac k2(d-D)}
  2^{-\frac k2(i+1)D}(N\varepsilon_0\mu(Q))^{\frac 12}\\
&+C^{\frac 12}2^{\frac k2(d-1-D)}2^{-\frac k2(i+1)D}\mu(Q'')^{\frac 12}\Big)\\
&\le\frac 13(\varepsilon_0 N)^{\frac 12}C^{-\frac 12}2^{\frac k2} 
  2^{-\frac k2(i+1)D}\mu(Q)^{\frac 12}+\frac 13 2^{-\frac k2(i+1)D}
  \mu(Q'')^{\frac 12}. 
\end{split}
\end{equation}

Combining \eqref{eq:nonporoestimate} with \eqref{eq:poroestimate}, we 
continue using Lemma \ref{lemma:worstcase} and \eqref{trivialestimate}
\begin{equation*}
\begin{split}
&\sum_{Q''\in\mathcal{Q}_I}\sum_{Q'\prec Q''}\beta(Q'')\beta(Q')
  r_{Q'}^{\frac D2}\mu(Q')^{\frac 12}\\
&\le\frac 13 C^{-\frac 12}2^{-\frac k2(d-1-D)}\sum_{Q''\in\mathcal{Q}_I}
  \Big(\sum_{\substack{Q' \prec Q''\\Q'\cap Q''_{por}=\emptyset}}
  \beta(Q')r_{Q'}^{\frac D2}\mu(Q')^{\frac 12}\\
&\phantom{eeeeeeeeeeeeeeeeeeeeeeeee}
  +\sum_{\substack{Q' \prec Q''\\Q'\subset Q''_{por}}}
  \beta(Q')r_{Q'}^{\frac D2}\mu(Q')^{\frac 12}\Big)\\
&\le\frac 13 C^{-\frac 12}2^{-\frac k2(d-1-D)}\sum_{Q''\in\mathcal{Q}_I}
  \Big(\frac13(\varepsilon_0 N)^{\frac 12}C^{-\frac 12}2^{\frac k2} 
  2^{-\frac k2(i+1)D}\mu(Q)^{\frac 12}\\
&\phantom{eeeeeeeeeeeeeeeeeeeeeeee}+\frac23 2^{-\frac k2(i+1)D}
  \mu(Q'')^{\frac 12}\Big)\\
&\le R(\varepsilon_0)r_{Q}^{\frac D2}\mu(Q)^{\frac 12}+\frac 29 C^{-\frac 12}
  2^{-\frac k2(d-1-D)}\sum_{Q''\in\mathcal{Q}_I}2^{-\frac k2(i+1)D}
  \mu(Q'')^{\frac 12}\\
&\leq R(\varepsilon_0)r_{Q}^{\frac D2}\mu(Q)^{\frac 12}+\frac 29 C^{-\frac 12}
  2^{-\frac k2(d-1-D)}2^{\frac{kd}2}2^{-\frac k2(i+1)D}\mu(Q)^{\frac 12}\\ 
&\leq\big(R(\varepsilon_0)+\frac 29 C^{-\frac 12}2^{\frac k2}\big)
  r_{Q}^{\frac D2}\mu(Q)^{\frac 12}\\
&=\frac12 C^{-\frac 12}2^{\frac k2}r_{Q}^{\frac D2}\mu(Q)^{\frac 12}.    
\end{split}
\end{equation*}

Recalling that $C\geq 2d2^l$, it is evident that 
$\#\mathcal{Q}_B\le C2^{k(d-1)}$. From \eqref{trivialestimate} and
Lemma \ref{lemma:worstcase} we obtain
\begin{align*}
&\sum_{Q''\in\mathcal{Q}_B}\sum_{Q'\prec Q''} 
 \beta(Q')\beta(Q'')r_{Q'}^{\frac D2}\mu(Q')^{\frac 12}\\
&\le\frac 19 C^{-1}2^{-k(d-1-D)}
 \big(C2^{2k(d-\frac 12)}\big)^{\frac 12}2^{-\frac k2(i+2)D}\mu(Q)^{\frac 12}\\
&=\frac 19 C^{-\frac 12}2^{\frac k2}r_{Q}^{\frac D2}\mu(Q)^{\frac 12}.
\end{align*}
Putting together the above estimates, proves \eqref{claim1} when $n=2$.

Next we assume that \eqref{claim1} holds when $n=m$ and verify it in
the case $n=m+1$. Again, we divide $Q$ into two regions of subcubes
$\mathcal{Q}_B, \mathcal{Q}_I \subset \mathcal{Q}^{i+1}$ defined above. 
Lemma \ref{lemma:boundary} is used to evaluate the sum in the
interior region $\mathcal{Q}_I$ whereas the induction hypothesis is applied
in the boundary region $\mathcal{Q}_B$. Let 
$R(\varepsilon)=R(\varepsilon,m+1)$ and $\varepsilon_0$ be as in \eqref{eq:R1} 
and \eqref{eq:R2}. 

In $\mathcal{Q}_I$ we begin with
\begin{align*}
&\sum_{ Q''' \in\mathcal{Q}_I}\sum_{Q' \prec_m Q'''}\bigl(\prod_{j=1}^{m+1}
 \beta(Q'_{i+j})\bigr)r_{Q'}^{\frac D2}\mu(Q')^{\frac 12}\\ 
&=\sum_{Q''' \in\mathcal{Q}_I}\sum_{Q''\prec_{m-1} Q'''}\prod_{j=1}^{m}
 \beta(Q''_{i+j})\sum_{Q'\prec Q''}\beta(Q')r_{Q'}^{\frac D2}
 \mu(Q')^{\frac 12}.  
\end{align*} 
From \eqref{eq:nonporoestimate} and \eqref{eq:poroestimate} we get
for the inner sum 
\begin{align*}
&\sum_{Q'\prec Q''}\beta(Q')r_{Q'}^{\frac D2}\mu(Q')^{\frac 12}\\
&\leq\frac 13 (\varepsilon_0 N)^{\frac 12}C^{-\frac 12}2^{\frac k2} 
  2^{-\frac{kmD}2}r_{Q}^{\frac D2}\mu(Q)^{\frac 12}+ 
  \frac 23 r_{Q''}^{\frac D2}\mu(Q'')^{\frac 12}.
\end{align*}
Repeating this $m$ times and using \eqref{eq:R2}, \eqref{trivialestimate} and 
Lemma \ref{lemma:worstcase} yields
\begin{equation}\label{eq:interior}
\begin{split} 
&\sum_{Q''' \in\mathcal{Q}_I}\sum_{Q'\prec_m Q'''}\bigl(\prod_{j=1}^{m+1}
  \beta(Q'_{i+j})\bigr)r_{Q'}^{\frac D2}\mu(Q')^{\frac 12}\\ 
&\leq R(\varepsilon_0)r_{Q}^{\frac D2}\mu(Q)^{\frac 12}+\big(\frac 23\big)^m
  \sum_{Q'''\in\mathcal{Q}_I}\beta(Q''')r_{Q'''}^{\frac D2}
  \mu(Q''')^{\frac 12}\\
&\leq R(\varepsilon_0)r_{Q}^{\frac D2}\mu(Q)^{\frac 12}\\
&+\frac 13\big(\frac 23\big)^m C^{-\frac 12}2^{-\frac k2(d-1-D)}
  2^{\frac{kd}2}2^{-\frac k2(i+1)D}\mu(Q)^{\frac 12}\\  
&\leq\big(R(\varepsilon_0)+\frac 13\big(\frac 23\big)^{m}
  C^{-\frac 12}2^{\frac k2}\big)r_{Q}^{\frac D2}\mu(Q)^{\frac 12}\\   
&\leq\frac 12 C^{-\frac 12}2^{\frac k2}r_{Q}^{\frac D2}\mu(Q)^{\frac 12}.    
\end{split}
\end{equation}
Notice that $R(\varepsilon_0,i)<R(\varepsilon_0,m+1)$ for all $i=2,\ldots,m$ by
\eqref{eq:R2}.

Finally, the induction hypothesis, Lemma \ref{lemma:worstcase} and 
\eqref{trivialestimate} combine to give in the boundary region $\mathcal{Q}_B$
(recall \eqref{defC})
\begin{equation}\label{eq:boundary}
\begin{split}
&\sum_{ Q'' \in\mathcal{Q}_B}\sum_{Q' \prec_m Q''}\bigl(\prod_{j=1}^{m+1}
  \beta(Q'_{i+j})\bigr)r_{Q'}^{\frac D2}\mu(Q')^{\frac 12}\\  
&\le C^{-\frac 12}2^{\frac k2}\sum_{Q'' \in\mathcal{Q}_B}\beta(Q'')
   r_{Q''}^{\frac D2}\mu(Q'')^{\frac 12}\\ 
&\leq\frac 13 C^{-1}2^{\frac k2}2^{-\frac k2(d-1-D)}
   \big(C2^{k(d-1)}\big)^{\frac 12}2^{-\frac k2(i+1)D}\mu(Q)^{\frac 12}\\ 
&=\frac 13 C^{-\frac 12}2^{\frac k2}r_{Q}^{\frac D2}\mu(Q)^{\frac 12}. 
\end{split}
\end{equation}
The claim follows by summing up \eqref{eq:interior} and
\eqref{eq:boundary}. 
\end{proof}

In the previous lemma $p$ played no role. Our next result quantifies the fact 
that we really gain something if a large proportion of scales is porous.

\begin{lemma}\label{lemma:prod}
Let $0<p<1$ and 
\[
D>d-p+\frac{\log (9C)}{k\log 2 }.
\]
Suppose that $\mu$ is a Radon measure on $\mathbb R^d$. Then there are 
$n\in\mathbb N$ and $K>1$ such that
\[
\prod_{j=0}^{L-1}C^{\frac 12}2^{-\frac k2}\prod_{i=jn+1}^{(j+1)n}
  \beta(Q_x^i)\ge K^L
\] 
for all $L\in\mathbb N$ and $x\in\mathbb R^d$ with 
\[
\#\{0\le j\le nL-1\,:\,\por(\mu,x,2^{-kj+l},\varepsilon)\ge\alpha\}\ge pnL
\]
for some $\varepsilon>0$. Here $n$ and $K$ are independent of $x$ and $L$.
\end{lemma}

\begin{proof}
Let 
\[
D_0=d-p+\frac{\log (9C)}{k \log 2 }\text{ and }\,
\delta=(\frac 1{D_0}-\frac 1D)(d-p)>0. 
\]
Choose $n\in\mathbb N$ such that
\begin{equation}\label{eq:n}
2^{k\delta Dn}>C^{-1}2^{k}.
\end{equation}
Obviously, $Q_{x}^j$ is porous whenever
$\por(\mu,x,2^{-k(j-1)+l},\varepsilon)\ge\alpha$, and thus 
\[
\#\{1\leq j\leq nL\,:\,Q_{x}^{j}\text{ is porous}\}\ge pnL.
\] 
By \eqref{weights} we have 
\begin{align*}
&\prod_{j=0}^{L-1}C^{\frac 12}2^{-\frac k2}
  \prod_{i=jn+1}^{(j+1)n}\beta(Q_x^i)\\
&=\big(C^{\frac 12}2^{-\frac k2}\big)^L\prod_{i=1}^{Ln}\beta(Q_x^i)\\
&\geq\big(C^{\frac 12}2^{-\frac k2}\big)^L\big(\frac 13\big)^{Ln}
  \big(C^{-\frac 12}2^{-\frac k2(d-1-D)}\big)^{pLn}
  \big(2^{-\frac k2(d-D)}\big)^{(1-p)Ln}\\
&=\Big(\big(C^{\frac 1D}2^{-\frac kD}\big)^L\big(3^{-\frac 2D}
  C^{-\frac pD}2^{-k(\frac{d-p}D-1)}\big)^{Ln}\Big)^{\frac D2}\\
&\ge\Big(\big(C^{\frac 1D}2^{-\frac{k}{D}}\big)^L2^{k\delta Ln}
  \big((9C)^{-\frac 1{D_0}}2^{-k(\frac{d-p}{D_0}-1)}\big)^{Ln}
  \Big)^{\frac D2}\\
&=\big(C2^{-k}2^{k\delta Dn}\big)^{\frac L2}       
\end{align*}
which gives the claim by \eqref{eq:n}.
\end{proof}

\begin{remark}\label{scalechange}
Mean porosity is defined in terms of scales $2^{-i}$ where $i\in\mathbb{N}$. 
However, in the proof of Theorem
\ref{theorem:pdimofporo} we need to use scales of the form $2^{-ki+m}$ where
$k,m\in\mathbb{N}$. This problem may be overcome by the following observation: 
Let $\varepsilon>0$ and $x\in\mathbb R^n$ such that 
$\#\{1\leq j\leq Nk\,:\,\por(\mu,x,2^{-j},\varepsilon)\geq \alpha\}\geq pNk$.
Then there is an integer $t$ with $0\le t\le k-1$ such that 
\[
\#\{1\leq j\leq N\,:\,\por(\mu,x,2^{-kj+t},\varepsilon)\geq \alpha\}\geq pN.
\]
Observe that the starting scale plays no role in Lemmas 
\ref{lemma:uniformestimate} and \ref{lemma:prod} and we choose it (out of
$k$ possibilities) depending on the point.
\end{remark}

\begin{proof}[Proof of Theorem \ref{theorem:pdimofporo}]
Suppose that $\mu$ is mean $(\alpha,p)$-porous for 
$\tfrac{15}{32}<\alpha<\tfrac12$ for which \eqref{eq:k} holds
and for $0\leq p\leq 1$. It is clearly enough to prove the claim for such 
$\alpha$'s. We shall prove that 
\begin{equation}\label{mainclaim}
\dimp\mu\le d-p+\frac{\log(9C)}{k\log 2}
\end{equation}
where $C=C(d)$ and $k=k(\alpha)$ are as in Lemma \ref{lemma:uniformestimate}. 
This implies the claim since, by \eqref{eq:kdef}, 
\[
\frac{\log(9C)}{k\log2}\leq\frac{C'(d)}{\log(\frac1{1-2\alpha})}.
\]
We may assume without loss of generality that $\mu(\mathbb{R}^d)=1$ and 
$\spt\mu\subset [0,1]^d$.

Let $0<p'<p$, $D>d-p'+\tfrac{\log(9C)}{k\log 2}$ and $K$ and $n$ be as in 
Lemma \ref{lemma:prod}. Take $\varepsilon_0$ as in 
\eqref{eq:R1} with $R(\varepsilon)=R(\varepsilon,n)$.
Since $\mu$ is mean $(\alpha,p)$-porous there are a Borel set 
$B\subset\mathbb R^d$ and constants $\eta>0$, $0<\varepsilon<\varepsilon_0$ 
and $I_0\in\mathbb N$ such that $\mu(B)>\eta$ and
\[
\#\{1\leq j\leq I\,:\,\por(\mu,x,2^{-j},\varepsilon)\ge\alpha\}>p'I
\] 
for all $I\ge I_0$ and for all $x\in B$. Let $N_0$ be so large that 
\begin{equation}\label{eq:betaestimate}
K^{N_0}>\frac{4k}\eta\text{ and }N_0kn\ge I_0.
\end{equation}

Let $\mathcal{Q}$ be any finite collection of $2^{kn}$-adic cubes with 
$r_Q=2^{-N_Qkn}\le 2^{-N_0 kn}$ for all $Q\in\mathcal{Q}$. For
those $Q\in\mathcal{Q}$ for which $Q\cap B\neq\emptyset$, we define
$\tau(Q)=\frac D2$. By Remark \ref{scalechange} there are 
$t=t(Q)\in\{0,\dots,k-1\}$ and $x\in Q$ such that 
\[
\#\{1\leq j\leq N_Qn\,:\,\por(\mu,x,2^{-kj+t+l},\varepsilon)\ge\alpha\}
\ge p'N_Qn
\]
provided that $Q\cap B\neq\emptyset$. (Note that $l$ is added here because it 
appears in the assumptions of Lemma \ref{lemma:prod} via the way we defined
porous cubes in the chapter after \eqref{eq:kdef}.)
For each such $Q\in\mathcal Q$ 
fix $Q'$ with $Q\subset Q'$ and 
$r_{Q'}=2^{-N_Qkn+t}$. In this way we obtain $k$ collections $\mathcal Q_m$ 
of $2^{kn}$-adic cubes scaled by the factor $2^m$, $m=0,\dots, k-1$.  
Since $N_Q\ge N_0$ we have by Lemma \ref{lemma:prod} and 
\eqref{eq:betaestimate} 
\[
r_{Q}^{\tau(Q)}\mu(Q)^{1-\frac{\tau(Q)}{D}}<\frac{\eta}{4k}\prod_{j=0}^{N_Q-1}
   C^{\frac 12}2^{-\frac k2}\prod_{i=jn+1}^{(j+1)n}\beta(Q_i)\,
   r_{Q}^{\frac D2}\mu(Q)^{\frac 12}
\]
for each $Q\in\mathcal Q_m$ where $\beta(Q_i)$ is as in \eqref{weights}. 
Summing over all $Q$'s and using \eqref{claim2}, we get for all 
$m=0,\dots, k-1$  
\begin{equation}\label{eq:B}
\begin{split}
&\sum_{\substack{Q\in\mathcal Q_m\\ Q\cap B\neq\emptyset}}
  r_{Q}^{\tau(Q)}\mu(Q)^{1-\frac{\tau(Q)}{D}}\\
&<\frac{\eta}{4k}\sum_{\substack{Q\in\mathcal Q_m\\ Q\cap B\neq\emptyset}}
  \prod_{j=0}^{N_Q-1}C^{\frac 12}2^{-\frac k2}\prod_{i=jn+1}^{(j+1)n}
  \beta(Q_i)\,r_{Q}^{\frac D2}\mu(Q)^{\frac 12}\\
&\leq\frac{\eta}{4k}r_{[0,1]^d}^{\frac D2}\mu(\mathbb{R}^d)^{\frac 12}
  =\frac{\eta}{4k}.
\end{split}
\end{equation}
Since 
$r_Q^{\frac D2}\mu(Q)^{\frac 12}\le r_{Q'}^{\frac D2}\mu(Q')^{\frac 12}$,
inequality \eqref{eq:B} implies that
\begin{equation}\label{eq:C}
\sum_{\substack{Q\in\mathcal Q\\Q\cap B\neq\emptyset}}
     r_{Q}^{\tau(Q)}\mu(Q)^{1-\frac{\tau(Q)}{D}}
\le\sum_{m=0}^{k-1}\sum_{\substack{Q\in\mathcal Q_m\\ Q\cap B\neq\emptyset}}
   r_{Q}^{\tau(Q)}\mu(Q)^{1-\frac{\tau(Q)}{D}}
\le k\frac\eta{4k}=\frac\eta 4.
\end{equation}

Suppose that $Q\in\mathcal{Q}$ with $Q\cap B=\emptyset$ and choose
$\tau(Q)>0$ so small that 
$r_{Q}^{\tau(Q)}\mu(Q)^{1-\frac{\tau(Q)}{D}}\leq\frac{1-\frac\eta 2}{1-\eta}
  \mu(Q)$. 
This choice is possible since
$r_{Q}^{\tau}\mu(Q)^{1-\frac{\tau}{D}}\rightarrow \mu(Q)$ as 
$\tau\downarrow 0$. Now
\begin{equation}\label{eq:noB}
\begin{split}
\sum_{\substack{Q\in\mathcal{Q}\\Q\cap B=\emptyset}}
   r_{Q}^{\tau(Q)}\mu(Q)^{1-\frac{\tau(Q)}{D}}
  \leq\frac{1-\frac\eta 2}{1-\eta}\mu(\mathbb{R}^d\setminus 
  B)< 1-\frac \eta 2.
\end{split}
\end{equation}
Combining \eqref{eq:C} and \eqref{eq:noB} gives
\[\sum_{Q\in\mathcal{Q}}r_{Q}^{\tau(Q)}\mu(Q)^{1-\frac{\tau(Q)}{D}}<1-\frac\eta 4
<\mu(\mathbb{R}^d).\]
Since the same upper bound is valid also for infinite collections of cubes 
Lemma \ref{lemma:dimensionlemma} implies that $\dimp\mu\leq D$. Letting
$D\downarrow d-p+\tfrac{\log(9C)}{k\log 2}$ completes the proof of 
\eqref{mainclaim}. 
\end{proof}

The following example shows that we cannot get a better upper 
estimate on the dimension than
\[ 
d - p + p\frac{C}{\log(\frac{1}{1-2\alpha})}
\]
even for mean porous sets. For simplicity, we consider the case $d=1$.
We use the notation $\dimH$ for Hausdorff dimension. The idea of the example
is the following: First we take $mnl$ steps in the construction of a standard 
Cantor set where each interval is substituted by two subintervals of length 
$2^{-nk}$. This will give us $(mnl-1)nk$ scales where porosity is at least
$\frac 12-2^{-kn}$. Next we divide each construction interval into 
$2^{(nk-mn)nkl}$ subintervals and choose every other of them. This guarantees
that the dimension will be large enough but the cost of it is the term $-nk$
in the number of porous scales. To get rid of this we iterate this procedure
increasing $l$ at every step.

\begin{example}\label{asymptoticallybest}
Fix $p=\frac mk\in ]0,1[\cap\mathbb Q$. Define for all $n\in\mathbb N$ a 
Cantor-type mean-porous set  
\[
A_{p,n}=\bigcap_{j=1}^\infty\bigcup_{\substack{g_i\in\mathcal S_i\\ i = 1,\dots, j}}
        (g_1\circ\cdots\circ g_j)([0,1]),
\]
with
\[
\mathcal S_l=\{K_{j_1}\circ\cdots\circ K_{j_{mnl}}\circ f_i\, :\, 
   i\in\{0,\dots,2^{(nk-mn)nkl-1}-1\} ~,~ j_t\in\{1,2\}\, \forall\, t\},
\]
where
\begin{align*}
f_i & :\mathbb R\to\mathbb R : x\mapsto 2^{-(nk-mn)nkl}x + i2^{-(nk-mn)nkl+1},\\
K_1 & : \mathbb{R} \to \mathbb{R} : x \mapsto 2^{-nk}x \text{ and}\\
K_2 & : \mathbb{R} \to \mathbb{R} : x \mapsto 2^{-nk}x + 1-2^{-nk}.
\end{align*}
A standard calculation shows that
\[
\dimH A_{p,n}=1-\frac mk+\frac m{nk^2}.
\]
On scales $2^{-s}$, where
\[
s-\frac{l^2+l}2(nk)^2\in\{0,\dots,(mnl-1)nk-1\}\text{ for some }l\in\mathbb N
\]
the porosity at each point $x\in A_{p,n}$ is at least $\frac 12-2^{-kn}$. Hence,
$A_{p,n}$ is mean $(\frac 12-2^{-kn},p)$-porous. Finally, 
we notice that
\[ 
1-p+ p\frac{\tfrac 12\log 2}{\log(\frac 1{1-2\alpha})}=1-\frac mk
   +\frac m{2k(kn-1)}\le\dimH A_{p,n}.
\]
\end{example}

\section{Mean porous measures are not necessarily approximable by 
            mean porous sets}\label{counterexample}

The aim of this section is to clarify relations between mean porosities of
sets and those of measures by verifying that, contrary to 
\cite[Proposition 1]{BS}, mean porous measures cannot be approximated by 
mean porous sets. For simplicity, we restrict our consideration to 
$\mathbb R$. Our goal is to construct a mean porous measure $\mu$ on $[0,1]$ 
with the property that any mean porous set has zero $\mu$-measure. 

Intuitively, the measure $\mu$ is constructed in the following way: For all
positive integers $i$, we define slowly decaying weights $w(i)$ and start with
the uniform measure on the unit interval. On $i^{\text{th}}$ step of 
the construction we 
redistribute the measure on the dyadic intervals of length $2^{-i}$ by 
attaching weight $w(i)$ to one half of the dyadic interval of length 
$2^{-(i-1)}$ and $1-w(i)$ to
the other half. We alternate the half of the interval which gets most of the
measure. To be precise, define for all $0<p\leq 1$
\begin{align*}
  \beta_p(\mu)=\sup\{\alpha\ge0\,:\,&\mu(A)>0\text{ for some mean }\\
    &(\alpha,p)\text{-porous set } A\subset\mathbb R\}
\end{align*}
and proceed by modifying the construction of \cite[Example 4]{EJJ}: 
For positive integers $i$, set $w(i)=\tfrac 1{\log(i+2)}$ and
\begin{equation*}
s(i)=\begin{cases}
       w(i)&\text{ if $i$ is odd}\\
       1-w(i)&\text{ if $i$ is even}.
     \end{cases}
\end{equation*}

For all binary sequences $j_1\ldots j_i$, we denote by $I_{j_1\ldots j_i}$
the closed dyadic interval of length $2^{-i}$ whose left endpoint
in binary representation is $0,j_1j_2\ldots j_i$.
Let $\mu$ be the unique Radon probability measure on $[0,1]$ determined by 
the formula
\[
\mu(I_{j_1\ldots j_i})=\prod\limits_{k=1}^{i}(1-s(k))^{j_k}s(k)^{1-j_k}
\]
for all binary sequences $j_1\ldots j_i$. We change continuously the side of 
the larger weight since we want to avoid the technical difficulties caused by 
the fact that at some scales the small ball giving the porosity might be 
outside the dyadic interval containing the point we are considering.    

\begin{theorem}\label{thm}
The measure $\mu$ has the following properties:
\begin{enumerate}
\item $\beta_p(\mu)=0$ for all $0<p\le1$,\label{1}
\item $\mu$ is $\tfrac18$-porous and \label{2}
\item $\mu$ is mean $(\alpha,1)$-porous for all $0\le\alpha<\tfrac12$.\label{3}
\end{enumerate}
\end{theorem}

\begin{remark}\label{irrelevant} For (1) we need to take slowly decreasing 
weights $w(i)$ whilst for (2) and (3) we only need that 
$w(i)\to 0$ as $i\to\infty$. 
Note that in (2) the value of porosity is not optimal. Our aim is only 
to demonstrate that even porous measures need not be approximable by 
mean porous sets.

\end{remark}
      
Theorem \ref{thm} will be proven as a consequence of several lemmas. 
The first one, Lemma \ref{lemma:sll}, is a corollary of the strong law of 
large numbers. For $x\in[0,1]$, let $x_i$ denote the $i$:th digit of $x$ 
in the binary representation. The non-uniqueness of the representation plays 
no role here since $\mu$ is easily seen to be non-atomic.

\begin{lemma}\label{lemma:sll}
For $\mu$-almost all $x\in[0,1]$, we have
\[
\lim\limits_{i\rightarrow\infty}\frac{\#\{1\leq j\leq i\,
:\, x_j=x_{j+1}\}}{i}=0.
\]
\end{lemma}

\begin{proof}
Applying the strong law of large numbers, see \cite[Chapter X.7]{Fe} to the
sequence of mutually independent random variables $X_i=x_i$, if $i$ is odd,
and $X_i=1-x_i$, if $i$ is even, gives
\[
\lim_{i\rightarrow\infty}\frac{\#\{1\leq j\leq i\,:\,x_j=q(j)\}}{i}=1
\]
for $\mu$-almost all $x\in[0,1]$. Here $q(j)=j\mod2$. 
From this the claim follows easily.
\end{proof}

For all positive integers $j$, we use the notation $\mathcal D^j$ for the
half-open dyadic subintervals of $[0,1]$ having length $2^{-j}$.(We do not use 
the notation $\mathcal Q^j$ introduced in Section \ref{mainsection} since we 
want to distinguish between powers of $2$ and $2^k$.) Letting $m$ be a positive
integer and $E\subset[0,1]$,
we say that $D\in\mathcal D^j$ contains an $m$-hole of $E$ if there is
$D'\in\mathcal{D}^{j+m}$ such that $D'\subset D$ and $D'\cap E=\emptyset$.

\begin{lemma}\label{lemma:ballsandqubes} 
Suppose that $E\subset[0,1]$ is 
mean $(\alpha,p)$-porous for some $\alpha>0$ and $0<p\le 1$. Let
$m$ be a positive integer such that $2^{-m}<\tfrac \alpha 4$. For $x\in[0,1]$, 
let $D_x^j\in\mathcal D^j$ be such that $\{x\}=\bigcap_{j=1}^{\infty}D_x^j$. 
Then for $\mu$-almost all $x\in[0,1]$, we have
\[
\liminf\limits_{i\rightarrow\infty}\frac{\#\{1\leq j\leq i
\,:\,D_x^j\text{ contains an $m$-hole of $E$}\}}{i}>\frac p8.\]    
\end{lemma}

\begin{proof}
Observe that for all $x\in E$ there exists $N_1$ such that
\[
\#\{1\leq j\leq i\,:\,\por(E,x,2^{-j})\ge\alpha\}
>\frac{pi}2
\]
for all $i>N_1$. Lemma \ref{lemma:sll} implies that for $\mu$-almost every 
$x\in[0,1]$ we may choose a positive integer $N_2$ with $N_2>\tfrac{16}p$ such 
that
\[
\#\{1\leq j\leq i\,:\,x_j=x_{j+1}\}<\frac{pi}4
\]
for all $i>N_2$. Hence, for $\mu$-almost every $x\in[0,1]$ 
\begin{equation}\label{eq:comb}
\#\{3\leq j\leq i\,:\,\por(E,x,2^{-j})\ge\alpha\text{ and }
     x_{j-1}\neq x_j\}>\frac{pi}8
\end{equation}
for all $i>\max\{N_1,N_2\}$.

Consider $x\in[0,1]$ for which \eqref{eq:comb} holds and a positive integer
$j\ge3$ such that $x_{j-1}\neq x_j$ and 
$E$ is $\alpha$-porous for scale $j$ at $x$. An easy calculation yields that
$\mathopen]x-2^{-j},x+2^{-j}\mathclose[\subset D_x^{j-2}$. Since, by the choice 
of $m$, there is $D\in \mathcal{D}^{j+m-2}$ such that
$D\cap E=\emptyset$ and $D\subset\mathopen]x-2^{-j},x+2^{-j}\mathclose[$, 
it follows that $D_x^{j-2}$ contains an $m$-hole of $E$. Now \eqref{eq:comb} 
gives the claim.
\end{proof}

For the remaining two lemmas we need the following notation:
Let $l$ and $m$ be positive integers and $E\subset[0,1]$. 
If $Q\in\mathcal D^{(j+1)(l+m)}$ and 
$Q'\in\mathcal D^{j(l+m)}$ such that $Q\subset Q'$, we use the same notation 
$Q\prec Q'$ as in Section \ref{mainsection} with $k$ replaced by $l+m$. 
We say that $Q'\in\mathcal D^{j(l+m)}$ is porous with respect to $E$ if 
$Q\cap E=\emptyset$ for some $Q\prec Q'$. For $Q\in\mathcal D^{i(l+m)}$ and 
$j=0,\dots,i$, let $Q_j$ be the unique interval in $\mathcal D^{j(l+m)}$ for 
which $Q\subset Q_j$. Obviously, these definitions
depend on the indices $l$ and $m$, but in what follows this 
dependence does not cause any confusion and it is therefore omitted for the
sake of simplicity. We refer by $Q$ to $2^{l+m}$-adic intervals
and by $D$ to dyadic intervals. In particular for $D\in\mathcal D^i$
there exists unique intervals $D_j\in\mathcal D^j$ for all $j=0,\dots,i$.
  
Our next lemma is a consequence of the previous one.

\begin{lemma}\label{lemma:at}
Suppose that $A\subset[0,1]$ is mean $(\alpha,p)$-porous for some $\alpha>0$ 
and $0<p\le 1$ and $\mu(A)>0$. Let $m$ be a positive integer with 
$2^{-m}<\tfrac\alpha 4$. Then there are $p'>0$, positive integers $l$ and 
$N_0$ and $E\subset A$ with $\mu(E)>0$ such that 
for all positive integers $i\ge N_0$ and for all $Q\in\mathcal D^{i(l+m)}$
we have 
\[
\#\{0\leq j\leq i-1\,:\,Q_j\text{ is porous with respect to $E$ }\}\geq p'i.
\]
\end{lemma} 

\begin{proof} 
Lemma \ref{lemma:ballsandqubes} implies the existence of $E\subset A$ with 
$\mu(E)>0$ and that of a positive integer $N$ such that for all positive 
integers $i\ge N$ and for all $D\in\mathcal D_i$ with $D\cap E\ne\emptyset$ 
we have 
\begin{equation}\label{eq:akuankka}
   \#\{1\leq j\leq i\,:\,D_j\text{ contains an $m$-hole of $E$}\}>\frac{pi}8.
  \end{equation}  
It follows that \eqref{eq:akuankka} is valid for all $D\in\mathcal D^i$
when $i>N/(1-\tfrac p8)$. To see this, let $D\in\mathcal D^i$.
Denote by $j_0$ the largest $j$ with
$E\cap D_j\neq\emptyset$. If $j_0\geq N$, then the claim clearly holds
by \eqref{eq:akuankka}. On the other hand, assuming that $j_0<N$, we obtain
\[
\{1\leq j\leq i\,:\,D_j\text{ contains an $m$-hole of $E$}\}\geq i-j_0
\geq i-N>\frac{pi}8
\]
by the choice of $i$.
 
Choose a positive integer $l>N/(1-\tfrac p8)$ so large that 
$m<\tfrac p{16}(m+l)$. Dividing the sequence
$1,2,\dots,i(m+l)$ into successive blocks of length $m+l$, it follows that for
all positive integers $i$ and for all $D\in\mathcal D^{i(m+l)}$ we have
\begin{align*}
    \#\{1\leq j\leq i\,:&\,D_k\text{ contains an $m$-hole of $E$ for some}\\
    &(j-1)(m+l)\leq k\leq j(m+l)-m\}>\frac{pi}{16}.
  \end{align*}
Indeed, if this is not the case, then
\begin{align*}
  \#\{1\leq k\leq i(m+l)\,:\,&D_k\text{ contains an $m$-hole of $E$}\}\\
    &\leq\frac{pi}{16} l+im\leq\frac{pi}8(m+l)
\end{align*}
giving a contradiction to \eqref{eq:akuankka}. But if $D_k$ contains
an $m$-hole of $E$ for some $(j-1)(m+l)\leq k\leq j(m+l)-m$, then
$D_{(j-1)(m+l)}$ is porous with respect to $E$, and we obtain 
\[
\#\{0\leq j\leq i-1\,:\,Q_j\text{ is porous with respect to $E$}\}
>\frac{pi}{16}
\]
for all $Q\in\mathcal D^{i(m+l)}$.
\end{proof}

\begin{remark}\label{doublecount}
The use of $l$ and $m$ prevents us from counting same holes twice.  
\end{remark}

Next we introduce for each interval $Q$ a weight $\eta(Q)$ which is equal to
1 if $Q$ is not porous and 1 minus the portion of the measure of the hole
if $Q$ is porous. For this purpose, let $l$ and $m$ be positive integers 
and let $E\subset[0,1]$. Define 
\begin{equation*}
\eta(Q)
=\begin{cases}
1\text{ if }Q\text{ is not porous w.r.t. } E\\
1-w(i(l+m)+1)\times\cdots\times w((i+1)(l+m))\text{ otherwise}
\end{cases}
\end{equation*}
for all non-negative integers $i$ and $Q\in\mathcal D^{i(m+l)}$.

In formula \eqref{osa1} of Lemma \ref{lemma:estimates} we prove that 
the product of weights of all predecessors of $Q$ converges uniformly to 0 
as the size of $Q$ decreases. 
The reason for this is quite simple: on a positive percentage of scales there
are porous intervals $Q$ with $\eta(Q)<1$ and the product converges to 0
since the measure of the holes decreases very slowly. On the other hand,
the proof of formula \eqref{osa2} is based on the definition of $\eta$. It 
guarantees that we may utilize porous scales.  

\begin{lemma}\label{lemma:estimates}
Suppose that $A\subset[0,1]$ is mean $(\alpha,p)$-porous for some $\alpha>0$ 
and $0<p\leq1$ and $\mu(A)>0$. Let $p'>0$, $E\subset A$ and positive integers 
$m$, $l$ and $N_0$ be as in Lemma \ref{lemma:at}. Then for all positive 
integers $i\ge N_0$ and for all $Q\in\mathcal D^{i(m+l)}$ we have
\begin{equation}\label{osa1}
\prod\limits_{j=0}^{i}\eta(Q_j)\leq c(i)
\end{equation} 
where $c(i)$ is a constant depending only on $i$, $l$, $m$ and
$p'$ and $c(i)\rightarrow 0$ as $i\rightarrow\infty$. (The dependence on all
other parameters but $i$ is irrelevant for our purposes.) Moreover, 
\begin{equation}\label{osa2} 
\sum_{\substack{Q\in\mathcal D^{i(m+l)}\\Q\cap E\neq\emptyset}}
   \frac{\mu(Q)}{\prod\limits_{j=0}^{i-1}\eta(Q_j)} \leq \mu([0,1])=1 
\end{equation} 
for all positive integers $i$.
\end{lemma}
 
\begin{proof}
By decreasing $p'$, we may assume that $\tfrac 1{p'}$ is an integer such that 
$\tfrac 1{p'}>N_0$. (Note that in Lemma \ref{lemma:at} the integer $N_0$ is 
independent of $p'$.) Consider $i>\tfrac 1{p'}$. Letting $k$ be a positive 
integer with $\tfrac k{p'}<i\le\tfrac{k+1}{p'}$ and applying Lemma 
\ref{lemma:at} successively for the indices $\tfrac 1{p'}$, $\tfrac 2{p'}$, 
$\dots,\tfrac k{p'}$, gives
\begin{align*}
 \prod_{j=0}^i&\eta(Q_j)\le\prod_{q=1}^k
   \left(1-w((\tfrac{q}{p'}-1)(l+m)+1)\times\cdots\times
   w(\tfrac{q}{p'}(l+m))\right)\\
 =&\exp\left(\sum\limits_{q=1}^{k}\log\left(1-w((\tfrac{q}{p'}-1)(l+m)+1)
   \times\cdots\times w(\tfrac{q}{p'}(l+m))\right)\right)\\
 <&\exp\left(-\sum_{q=1}^{k}w((\tfrac{q}{p'}-1)(l+m)+1)\times\cdots
    \times w(\tfrac{q}{p'}(l+m))\right)\\
 <&\exp\left(-\sum_{q=1}^{k}\left(\log(\tfrac{q}{p'}
     (l+m)+2)\right)^{-(l+m)}\right)\xrightarrow[k\to\infty]{} 0, 
\end{align*}  
implying \eqref{osa1}. 

Inequality \eqref{osa2} is valid since
\begin{align*}
  \sum_{\substack{Q\in\mathcal D^{i(m+l)}\\Q\cap E\neq
   \emptyset}}\frac{\mu(Q)}{\prod\limits_{j=0}^{i-1}\eta(Q_j)}
 &=\sum_{\substack{Q'\in\mathcal D^{(i-1)(m+l)}\\Q'\cap E\neq\emptyset}} 
   \frac{1}{\prod\limits_{j=0}^{i-1}\eta(Q_j')}
   \sum_{\substack{Q\prec Q'\\Q\cap E\neq\emptyset}}\mu(Q)\\
 &\le\sum_{\substack{Q'\in\mathcal D^{(i-1)(m+l)}\\ 
    Q'\cap E\neq\emptyset}}\frac{\eta(Q')\mu(Q')}{\eta(Q')\prod\limits_{j=0}^{i-2}
    \eta(Q_j')}\\
 & =\sum_{\substack{Q'\in\mathcal D^{(i-1)(m+l)}\\ 
    Q'\cap E\neq\emptyset}}\frac{\mu(Q')}{\prod\limits_{j=0}^{i-2}\eta(Q_j')}.
\end{align*}
Here the inequality follows by dividing the inner summation into the two
possible cases: either $Q\cap E\ne\emptyset$ for all $Q\prec Q'$ or
$Q\cap E=\emptyset$ for some $Q\prec Q'$. 
\end{proof}

Now we are ready to prove the main theorem of this section. 
  
\begin{proof}[Proof of Theorem \ref{thm}]  
Assume contrary to the claim \eqref{1} that there is $\alpha>0$ and a mean 
$(\alpha,p)$-porous set $A\subset[0,1]$ with $\mu(A)>0$. Using Lemmas 
\ref{lemma:at} and \ref{lemma:estimates}, we find positive integers $m$ and
$l$ and $E\subset A$ with $\mu(E)>0$ satisfying the estimates \eqref{osa1} 
and \eqref{osa2}. Taking $N_0$ as in Lemma \ref{lemma:estimates}, we get for
all $i\ge N_0$
\begin{align*}
 \frac{\mu(E)}{c(i)}\le\sum_{\substack{Q\in\mathcal D^{(i+1)(m+l)}\\ 
   Q\cap E\neq\emptyset}}\frac{\mu(Q)}{\prod\limits_{j=0}^i\eta(Q_j)}\le 1,
\end{align*}
giving $\mu(E)\leq c(i)\to 0$ as $i\to\infty$ contrary to $\mu(E)>0$. This 
proves \eqref{1}.
  
An elementary calculation proves \eqref{2}, for details see 
\cite [Example 4]{EJJ}. Clearly, $\frac 18$ is not the best 
possible value, but this is irrelevant for our purposes.

For the remaining part (3), given $0<\alpha<\frac 12$, fix a positive integer
$n$ such that $2^{-n}<\frac 12-\alpha$.
As in Lemma \ref{lemma:sll}, we obtain 
\[
\lim_{i\to\infty}\frac{\#\{1\leq j\leq i\, :\, x_k=x_{k+1}\text{ for some }
  j\le k<j+n\}}{i}=0
\]
for $\mu$-almost all $x\in\mathopen[0,1\mathclose]$. 
Thus, for $0<p<1$, we find for $\mu$-almost all 
$x\in\mathopen[0,1\mathclose]$ a positive integer $N$ such that
\[ 
\{1\leq j\leq i\, :\, x_k=x_{k+1}\text{ for some } j\le k<j+n\}<(1-p)i
\]
for all $i>N$. This implies that $\mu$ is mean $(\alpha,p)$-porous at 
$\mu$-almost all points $x\in\mathopen[0,1\mathclose]$ (for details, see 
\cite [Example 4]{EJJ}), and consequently (3) holds. 
\end{proof}

\end{document}